\documentclass[a4paper,reqno]{amsart}%
\usepackage{pgf,tikz}
\usepackage{amssymb}
\usepackage{amsmath}
\usepackage{amsfonts}
\usepackage{graphicx}
\usepackage{bm}%
\providecommand{\U}[1]{\protect\rule{.1in}{.1in}}
\usetikzlibrary{backgrounds}
\usetikzlibrary{arrows}
\let\oldmathbf\mathbf
\renewcommand{\mathbf}[1]{\boldsymbol{\oldmathbf{#1}}}
\newtheorem{theorem}{Theorem}

\newtheorem{corollary}[theorem]{Corollary}

\newtheorem{definition}[theorem]{Definition}

\newtheorem{lemma}[theorem]{Lemma}

\newtheorem{proposition}[theorem]{Proposition}

\allowdisplaybreaks
\begin{document}
\title{Optimal asymptotic bounds for designs on manifolds}
\author[B. Gariboldi]{Bianca Gariboldi}
\address{Dipartimento di Ingegneria Gestionale, dell'Informazione e della Produzione,
Universit\`a degli Studi di Bergamo, Viale Marconi 5, Dalmine BG, Italy}
\email{biancamaria.gariboldi@unibg.it}
\author[G. Gigante]{Giacomo Gigante}
\address{Dipartimento di Ingegneria Gestionale, dell'Informazione e della Produzione,
Universit\`a degli Studi di Bergamo, Viale Marconi 5, Dalmine BG, Italy}
\email{giacomo.gigante@unibg.it}


\begin{abstract}
We extend to the case of a $d$-dimensional compact connected oriented  Riemannian manifold $\mathcal M$ the theorem of A. Bondarenko, D. Radchenko and M. Viazovska \cite{BRV} on the existence of $L$-designs consisting of $N$ nodes, for any  $N\ge C_{\mathcal M} L^d$. For this, we need to prove a version of the Marcinkiewicz-Zygmund inequality for 
the gradient of diffusion polynomials.
\end{abstract}
\maketitle

\section{Introduction}

Let $\mathcal{M}$ be a connected compact orientable $d$-dimensional Riemannian
manifold without boundary with normalized Riemannian measure $d\mu$, such that
$\mu\left(  \mathcal{M}\right)  =1$. We shall denote the Riemannian distance
between $x$ and $y$ by $\left\vert x-y\right\vert $. Let $\left\{  \varphi
_{k}\right\}  _{k=0}^{+\infty}$ be the eigenfunctions of the (positive)
Laplace-Beltrami operator, with eigenvalues $0=\lambda_{0}^{2}<\lambda_{1}%
^{2}\leq\lambda_{2}^{2}\leq\ldots$ , $\Delta\varphi_{k}=\lambda_{k}^{2}%
\varphi_{k}$.

The space of diffusion polynomials of bandwith $L\geq0$ is%
\[
\Pi_{L}=\operatorname*{span}\left\{  \varphi_{k}:\lambda_{k}\leq L\right\}  .
\]
We say that a set of points $\left\{  x_{j}\right\}  _{j=1}^{N}\subset
\mathcal{M}$ is an $L$-design if
\[
\int_{\mathcal{M}}P\left(  x\right)  d\mu\left(  x\right)  =\sum_{j=1}%
^{N}\frac{1}{N}P\left(  x_{j}\right)  ,\quad\text{for all }P\in\Pi_{L}.
\]
Observe that since the above identity is trivially satisfied by constant
functions, and since by orthogonality of the eigenfunctions $\varphi_{k}$,%
\[
\int_{M}\varphi_{k}\left(  x\right)  d\mu\left(  x\right)  =0,\quad\text{for
all }k\geq1,
\]
then $\left\{  x_{j}\right\}  _{j=1}^{N}\subset\mathcal{M}$ is an $L$-design
if and only if
\[
\sum_{j=1}^{N}\frac{1}{N}P\left(  x_{j}\right)  =0,\quad\text{for all }P\in
\Pi_{L}^{0},
\]
where $\Pi_{L}^{0}$ is the subspace of $\Pi_{L}\subset L^2(\mathcal M,\,d\mu)$ orthogonal to the constant functions, that is
$\Pi_{L}^{0}=\operatorname*{span}\{\varphi_k:0<\lambda_k\le L\}$.

By Weyl's estimates on the spectrum of an elliptic operator \cite[Theorem
17.5.3]{HOR}, $\dim\left(  \Pi_{L}\right)  \sim L^{d}$. For each $L\geq0$
denote with $N\left(  L\right)  $ the minimal number of points in an
$L$-design in $\mathcal{M}$.

\begin{proposition}\label{sotto}
There exists a positive constant $c_{\mathcal{M}}$ such that $N\left(
L\right)  \geq c_{\mathcal{M}}L^{d}$ for every $L\geq0$.
\end{proposition}

\begin{proof}
Assume $\left\{  x_{j}\right\}  _{j=1}^{N\left(  L\right)  }$ is an $L$-design
with exactly $N\left(  L\right)  $ nodes. By Theorem~2.12 in \cite{quadrature}
there exists a constant $\beta>0$ such that for every $f$ in the Sobolev space
$W^{\alpha,1}\left(  \mathcal{M}\right)  $ with $\alpha>d$ one has%
\[
\left\vert \int_{\mathcal{M}}f\left(  x\right)  d\mu\left(  x\right)
-\sum_{j=1}^{N\left(  L\right)  }\frac{1}{N\left(  L\right)  }f\left(
x_{j}\right)  \right\vert \leq \beta L^{-\alpha}\left\Vert f\right\Vert
_{\alpha,1}.
\]
On the other hand, by Theorem 2.16 in \cite{quadrature}, there exists a
constant $\gamma>0$ such that for every $L$ there exists a function $f_{L}\in
W^{\alpha,1}\left(  \mathcal{M}\right)  $ with%
\[
\left\vert \int_{\mathcal{M}}f_{L}\left(  x\right)  d\mu\left(  x\right)
-\sum_{j=1}^{N\left(  L\right)  }\frac{1}{N\left(  L\right)  }f_{L}\left(
x_{j}\right)  \right\vert \geq \gamma N\left(  L\right)  ^{-\alpha/d}\left\Vert
f_{L}\right\Vert _{\alpha,1}.
\]
This gives
\[
N\left(  L\right)  \geq \gamma^{d/\alpha}\beta^{-d/\alpha}L^{d}.
\]

\end{proof}

Korevaar and Meyers \cite{KM} conjectured that when $\mathcal{M}$ is the
$d$-dimensional sphere, there is a constant $C_{d}$ such that $N\left(
L\right)  \leq C_{d}L^{d}$ for any positive $L$. Bondarenko, Radchenko and
Viazovska \cite{BRV} show an even stronger version of Korevaar and Meyer's
conjecture, namely they show that there is a constant $C_{d}$ such that
\textit{for every} $N\geq C_{d}L^{d}$ there exists an $L$-design in
the $d$-dimensional sphere with exactly $N$ nodes. Later, Etayo, Marzo and
Ortega-Cerd\`{a} \cite{EMOC} by means of the same techniques as in \cite{BRV},
generalize the result of Bondarenko, Radchenko and Viazovska to the case of an
affine algebraic manifold. In particular the main ingredients in these proofs
are a result from the Brouwer degree theory, a partition of the ambient space
$\mathcal{M}$ into equal area regions with small diameter, and a
Marcinkiewicz-Zygmund inequality for \textit{the gradient} of diffusion
polynomials. The result from the Brouwer degree theory is completely abstract
and can be directly applied independently of the manifold.

\begin{theorem}\label{brouwer}
\cite[Theorem 1.2.9.]{brouwer} Let $H$ be a finite dimensional Hilbert space with inner product $\langle\cdot,\cdot\rangle$. Let $f:H\rightarrow H$
be a continuous mapping and $\Omega$ an open bounded subset with boundary
$\partial\Omega$ such that $0\in\Omega\subset H$. If $\left\langle
x,f\left(  x\right)  \right\rangle >0$ for all $x\in\partial\Omega$, then
there exists $x\in\Omega$ satisfying $f\left(  x\right)  =0.$
\end{theorem}

The second result has been proved recently in \cite{GL} for Ahlfors regular
metric measure spaces, and therefore holds for compact Riemannian manifolds as well.

\begin{theorem}
\label{partition}There exist two positive constants $c_1$ and $c_{2}$ such that
for every $N\ge1$ there is a collection of sets $\mathcal{R}=\left\{
R_{j}\right\}  _{j=1}^{N}$ that partition $\mathcal{M}$ in the sense that
$\cup_{j=1}^N R_j=\mathcal M$, and $\mu(R_i\cap R_j)=0$ for all $1\le i<j\le N$,
such that each region has measure $1/N$, is contained in a geodesic (closed) 
ball $X_j$ of radius $c_{2}N^{-1/d}$ and contains a geodesic ball $Y_j$ of radius $c_1 N^{-1/d}$.
\end{theorem}

The Marcinkiewicz-Zygmund inequality for diffusion polynomials on manifolds
(and much more general spaces) has been proved in a series of papers by
Maggioni, Mhaskar and Filbir \cite{FM2010,FM2011, MM}.

\begin{theorem}
\label{MZ polys}\cite[Theorem 5.1]{FM2011} Assume that $c_1$ and $c_2$ are constants
for which Theorem \ref{partition} holds. Then there exists a constant $C>0$ such
that for all integers $N\ge 1$, for all partitions $\mathcal{R=}\left\{
R_{j}\right\}  _{j=1}^{N}$ with constants $c_1$ and $c_2$ as in Theorem \ref{partition}, for all $x_{j}\in
R_{j}$, for all $L\le N^{1/d}$ and for all $P\in\Pi_{L}$%
\[
\left\vert \int_{\mathcal{M}}\left\vert P\left(  x\right)  \right\vert
d\mu\left(  x\right)  -\sum_{j=1}^{N}\frac{1}{N}\left\vert P\left(
x_{j}\right)  \right\vert \right\vert \leq CLN^{-1/d}\int_{\mathcal{M}%
}\left\vert P\left(  x\right)  \right\vert d\mu\left(  x\right)  .
\]

\end{theorem}

When $\mathcal{M}$ is the sphere and the diffusion polynomials are
restrictions to $\mathcal{M}$ of polynomials in $d+1$ real variables of
degree at most $L$, then the gradient of a polynomial is again a polynomial,
and therefore the Marcinkiewicz-Zygmund inequality for gradients follows
easily from the above Theorem \ref{MZ polys}. In the case of algebraic
manifolds, ad hoc arguments that use the complexification of the variety
$\mathcal{M}$ can be applied (see \cite{EMOC}). In the general case of
Riemannian manifolds, the above types of arguments fail. Here we show a
Marcinkiewicz-Zygmund inequality for gradients of diffusion polynomials and
consequently prove the Korevaar and Meyer's conjecture in the case of
Riemannian manifolds.

\begin{theorem}
\label{MZ gradients}  Assume that $c_1$ and $c_2$ are constants
for which Theorem \ref{partition} holds. Then there exists a constant $C_3=C_3(c_1,c_2)>0$ such
that for all integers $N\ge 1$, for all partitions $\mathcal{R=}\left\{
R_{j}\right\}  _{j=1}^{N}$ with constants $c_1$ and $c_2$ as in Theorem \ref{partition}, for all $x_{j}\in
R_{j}$, for all $L\le N^{1/d}$ and for all $P\in\Pi_{L}^{0}$,
\[
\left\vert \int_{\mathcal{M}}\left\Vert \nabla P\left(  x\right)  \right\Vert
d\mu\left(  x\right)  -\sum_{j=1}^{N}\frac{1}{N}\left\Vert \nabla P\left(
x_{j}\right)  \right\Vert \right\vert \leq C_3LN^{-1/d}\int_{\mathcal{M}%
}\left\Vert \nabla P\left(  x\right)  \right\Vert d\mu\left(  x\right)  .
\]

\end{theorem}

\begin{theorem}\label{final}
There exists a constant $C_{\mathcal{M}}$ such that for each $N\geq
C_{\mathcal{M}}L^{d}$ there exists an $L$-design in $\mathcal{M}$ with $N$ nodes.
\end{theorem}

In the proof of Proposition \ref{sotto} we mentioned Theorem 2.12 in \cite{quadrature}. This is a result on 
numerical integration for functions in Sobolev spaces and it says that 
if $\{x_j\}_{j=1}^N$ is an $L$-design on a compact Riemannian manifold $\mathcal M$, then for every $1\le p\le +\infty$
and for every $\alpha>d/p$ there exists a constant $\beta>0$ such that for every $f$ in the Sobolev space
$W^{\alpha,p}(\mathcal M)$
\[
\left|\int_{\mathcal M} f(x)d\mu(x)-\frac 1N\sum_{j=1}^Nf(x_j)\right|\le \beta L^{-\alpha}\|f\|_{\alpha,p}.
\]
By the above Theorem \ref{final}, if $\mathcal M$ is a connected compact oriented $d$-dimensional Riemannian manifold, 
for every positive integer $N$, setting $L=(N/C_{\mathcal M})^{1/d}$ there indeed exists an $L$-design on $\mathcal M$ 
consisting of $N$ nodes, and this immediately gives the following result on the worst case error in numerical integration
\begin{corollary}
For every $1\le p\le +\infty$
and for every $\alpha>d/p$ there exists a constant $\beta>0$ such that 
for every $N\ge 1$ there exists a collection of points $\{x_j\}_{j=1}^N$ such that for every $f\in W^{\alpha,p}(\mathcal M)$
\[
\left|\int_{\mathcal M} f(x)d\mu(x)-\frac 1N\sum_{j=1}^Nf(x_j)\right|\le \beta C_{\mathcal M}^{\alpha/d} N^{-\alpha/d}\|f\|_{\alpha,p}.
\]
\end{corollary}
By Theorem 2.16 in \cite{quadrature}, the exponent $-\alpha/d$ is best possible. This result should be compared 
with Corollary 6.3, Corollary 6.4 and Example 6.5 in \cite{Chen}, where the authors prove that if $1<p\le +\infty$ and $d/p<\alpha<d$, then 
a random choice of nodes $x_j\in R_j$ gives the desired decay rate $N^{-\alpha/d}$ for the worst 
case error in numerical integration if and only if $\alpha<d/2+1$. See also \cite{Brauchart1,Brauchart2} for previous results 
in the case of the sphere.

We wish to thank Luca Brandolini, Leonardo Colzani and Giancarlo Travaglini for several discussions on
this subject. 

\section{Introduction to Riemannian manifolds}
The following are well known facts about manifolds. The interested reader can find all details in
\cite{BGM, DC}.

A differentiable manifold of dimension $d$ is a set $\mathcal M$ and a family of injective 
maps $x_\alpha:U_\alpha\subset\mathbb R^d\to\mathcal M$ such that
\begin{enumerate}
\item The $U_\alpha$'s are open and $\cup_{\alpha}x_\alpha(U_\alpha)=\mathcal M$.
\item For any pair $\alpha,\beta$ with $x_\alpha(U_\alpha)\cap x_\beta(U_\beta)=W\neq\emptyset$,
the sets $x_\alpha^{-1}(W)$ and $x_\beta^{-1}(W)$ are open and the maps $x_\beta^{-1}\circ x_\alpha$
are $\mathcal C^\infty$.
\item The family $\{(U_\alpha,x_\alpha)\}$ is maximal relative to the above conditions.
\end{enumerate}
Each $(U_\alpha,x_\alpha)$ is called local chart, and a family $\{(U_\alpha,x_\alpha)\}$ satisfying (1) and (2) is called differentiable structure.

This induces a natural topology on $\mathcal M$. A set $A$ is open in $\mathcal M$ if and only if
$x_\alpha^{-1}(A\cap x_\alpha(U_\alpha))$ is open in $\mathbb R^d$ for all $\alpha$.
We will assume that with this topology $\mathcal M$ is a Hausdorff space with a countable basis.

We also say that a differentiable manifold $\mathcal M$ is orientable if 
\begin{itemize}
\item[(4)] For every pair $\alpha, \beta$ with $x_\alpha(U_\alpha)\cap x_\beta(U_\beta)=W\neq\emptyset$
the differential of the change of coordinates $x_\beta^{-1}\circ x_\alpha$ has positive determinant.
\end{itemize}

A map $f:\mathcal N\to\mathcal M$ between two differentiable manifolds is called differentiable in $p\in\mathcal N$ if for every local chart $(V,y)$ at $f(p)$ there exists a local chart $(U,x)$ at $p$ such that 
$f(x(U))\subset y(V)$ and the map $y^{-1}\circ f \circ x$ is $\mathcal C^{\infty}$. We say that $f$ is differentiable
in an open set of $\mathcal N$ if it is differentiable at all points of this open set.

A differentiable map $\alpha$ from the interval $(-\varepsilon,\varepsilon)\subset\mathbb R$ to $\mathcal M$ will be called a (differentiable) curve in $\mathcal M$. Suppose that $\alpha(0)=p$. Call 
$\mathcal D(\mathcal M)$ the set of functions on $\mathcal M$ differentiable in $p$. The tangent vector 
to $\alpha$ at $t=0$ is the function $\alpha'(0):\mathcal D(\mathcal M)\to \mathbb R$ given by
\[
\alpha'(0)f=\left.\frac{d(f\circ\alpha)}{dt}\right|_{t=0},\quad f\in\mathcal D(\mathcal M).
\]
A tangent vector at $p$ is the tangent vector at $t=0$ of a1 curve like the one above. The set of all
tangent vectors to $\mathcal M$ at $p$ will be indicated with $T_p\mathcal M$. This is a vector space of dimension $d$ called tangent space of $\mathcal M$ at $p$,
and the choice of a local chart $(U,x)$ around $p$ determines a natural basis of $T_p\mathcal M$ given by 
\[
\left\{\left(\frac{\partial}{\partial x_i}\right)_{p}\right\}_{i=1}^d
\]
where setting $x^0=x^{-1}(p)$, $(\partial/\partial x_i)_{p}$ is the tangent vector to the curve 
$x(x^0_1,\ldots,x^0_i+t,\ldots,x^0_d)$ at $t=0$.
The set $T\mathcal M=\{(p,v):p\in\mathcal M,\,v\in T_p\mathcal M\}$ can be given a differentiable structure 
that makes it a differentiable manifold of dimension $2d$ called tangent bundle. The local charts are 
$(U_\alpha\times \mathbb R^d, y_\alpha)$, where 
\[
y_\alpha(x^\alpha_1,\ldots,x^\alpha_d,u_1,\ldots,u_d)=(x_\alpha(x^\alpha_1,\ldots,x^\alpha_d),
\sum_{i=1}^du_i\frac{\partial}{\partial x_i^\alpha}).
\]
A vector field $X$ on a differentiable manifold $\mathcal M$ is a map from $\mathcal M$
to the tangent bundle $T\mathcal M$ such that $X(p)=(p,v)$ for some $v\in T_p\mathcal M$. We call $\mathcal X(\mathcal M)$
the space of differentiable vector fields on $\mathcal M$.

A Riemannian metric on a differentiable manifold $\mathcal M$ is a correspondence which associates to each point $p\in\mathcal M$ an inner product $\langle\cdot,\cdot\rangle_p$ (that is, a symmetric, bilinear, positive definite form) on the tangent space $T_p\mathcal M$ which varies differentiably in the sense that 
if $(U,x)$ is a local chart around $p$ and if $q=x(x_1,\ldots,x_d)$ then 
\[
g_{ij}(x_1,\ldots,x_d)=\left\langle\left(\frac\partial{\partial x_i}\right)_q,\left(\frac\partial{\partial x_j}\right)_q\right\rangle_q
\]
is a differentiable function on $U$. A differentiable manifold with a given Riemannian metric will be called 
a Riemannian manifold. If $v$ is a tangent vector to $\mathcal M$ at $p$, we set $\|v\|^2_p=\langle v,v\rangle_p$. The length of a differentiable curve $\alpha$ from the interval $[a,b]$ to $\mathcal M$
is defined as
\[
\int_a^b\|\alpha'(t)\|_{\alpha(t)}dt.
\]

We define the distance $|p-q|$ between two points in a Riemannian manifold, $p,\,q\in\mathcal M$, as 
the infimum of the lengths of all the differentiable curves joining $p$ and $q$. This is indeed a distance,
and it turns $\mathcal M$ into a metric space that has the same topology as the manifold's natural topology.

If $\mathcal M$ is a compact Riemannian manifold, then for any two points $p$ and $q$ in $\mathcal M$
there exists at least one differentiable curve $\alpha$ joining $p$ and $q$ that realizes the infimum of the lengths 
of all the differentiable curves joining $p$ and $q$. Furthermore, the covariant derivative of $\alpha'$ along $\alpha$ equals zero
(curves that satisfy this property are called {\it geodesics}). We refer the reader to \cite{DC} for the precise definition of covariant derivative,
here it suffices to recall that if $\alpha$ is a geodesic then $\|\alpha'(t)\|_{\alpha(t)}$ is constant, and one can normalize $\alpha$ in such a way that
$\|\alpha'(t)\|_{\alpha(t)}=1$.

Let $(U,x)$ be a local chart where the metric has local representation given by  $g_{ij}(x_1,\ldots,x_d)$. The positive measure $d\nu_{(U,x)}=\sqrt{\det (g_{ij})}dx_1\ldots dx_d$ on $U$ induces a positive
measure $d\mu_{(U,x)}$ on $x(U)$ given by
\[
\int_{x(U)} fd\mu_{(U,x)}=\int_U f\circ x d\nu_{(U,x)}.
\]
If $(U,x)$ and $(V,y)$ are two local charts with $U\cap V=W\neq\emptyset$, then one can show that
$d\mu_{(U,x)}$ coincides with $d\mu_{(V,y)}$ on $W$. By a standard partition of unit argument, there exists 
a unique measure $d\mu$ on $\mathcal M$ that coincides with $d\mu_{(U,x)}$ on $U$
for all local charts $(U,x)$. This measure is called canonical measure on $\mathcal M$.

It can be shown that there exist two positive constants $c_4$ and $c_5$ such that for any point $x\in\mathcal M$ and for any radius $r\le\mathrm{diam}(\mathcal M)$, the measure of the ball 
$B(x,r)=\{y\in\mathcal M:|x-y|<r\}$ satisfies  the inequalities
\begin{equation}
\label{bolle}
c_4 r^d\leq\mu(B(x,r))\leq c_5 r^d.
\end{equation}
It follows easily that there exists a positive constant $c_6$ such that if $f:[0,+\infty)\to[0,+\infty)$ is a decreasing function and if $x\in\mathcal M$ then 
\begin{equation}
\label{radiale}
\int_{\mathcal M}f(|x-y|)d\mu(y)\le c_6\int_{0}^{+\infty}f(t)t^{d-1}dt.
\end{equation}	

For any $f\in\mathcal D(\mathcal M)$ we define the gradient of $f$ as a vector field $\nabla f$ on the Riemannian manifold $\mathcal M$ given by 
\[
\langle\nabla f(p),v\rangle_p=vf,\quad p\in\mathcal M,\,v\in T_p\mathcal M.
\]
In local coordinates, the gradient is given by the formula
\[
\nabla f(p)=\sum_{j=1}^d\sum_{i=1}^dg^{ij}\left(\frac{\partial }{\partial x_i} f\right)\frac{\partial }{\partial x_j},
\]
where $g^{ij}$ are the entries of the inverse matrix of $g_{ij}$. It follows from the definition that 
\[
\|\nabla f\|_p=\sup Xf
\]
where the supremum is taken over all the differentiable vector fields $X$ with norm $\|X\|_p\le 1$.

If $f\in\mathcal D(\mathcal M)$ and $\alpha:[0,|p-q|]\to \mathcal M$ is a normalized geodesic joining $p$ and $q$, then 
\begin{align}\label{geodesic}
\nonumber
 |f(p)-f(q)|&=\left|\int_0^{|p-q|}\frac{d}{dt}(f(\alpha(t)))dt\right|=\left|\int_{0}^{|p-q|}\langle\nabla f(\alpha(t)),\,\alpha'(t)\rangle dt\right|\\
&\le |p-q|\sup_{t}\|\nabla f(\alpha(t))\|. 
\end{align}

One can also define the divergence of a differentiable vector field $X$. To keep this presentation simple, we only 
give the definition in local coordinates, that is
\[
\mathrm{div} (X)=\frac{1}{\sqrt{\mathrm{det}(g_{ij})}}\sum_{k=1}^d\frac{\partial}{\partial x_k}\left(\sqrt{\mathrm{det}(g_{ij}})X_k\right),\quad X=\sum_{k=1}^d X_k\frac{\partial}{\partial x_k}.
\]

Finally, the Laplace-Beltrami operator $\triangle:\mathcal D(\mathcal M)\to\mathcal D(\mathcal M)$ is defined as
\[
\triangle f=-\mathrm{div}(\nabla f).
\]
Let $\mathcal M$ be a compact oriented Riemannian manifold and let $f\in\mathcal D(\mathcal M)$ and $X$ a differentiable vector field.  The following Green identity (see \cite[page 267]{L}) holds
\begin{equation}
\label{Green}
\int_{\mathcal M} \langle \nabla f, X\rangle d\mu=-\int_{\mathcal M}f \mathrm{div}(X)d\mu.
\end{equation}
Let now $\mathcal M$ be a connected compact Riemannian manifold. The Laplace-Beltrami operator 
$\triangle$ is self adjoint and positive definite. The eigenvalues of the Laplace-Beltrami operator
are the values $\lambda^2$ such that the equation 
\[
\triangle f=\lambda^2 f
\]
has a solution $f\in\mathcal D(\mathcal M)$, $f\neq 0$, called eigenfunction. These eigenvalues form a discrete 
sequence of non-negative numbers diverging to $+\infty$. The eigenfunctions associated with a given eigenvalue 
$\lambda^2$ form a finite dimensional subspace of $\mathcal D(\mathcal M)$ called eigenspace associated with $\lambda^2$. The dimension of the eigenspace will be called multiplicity of the eigenvalue. Different eigenspaces are orthogonal to each other in the Hilbert space $L^2(\mathcal M,d\mu)$,
and the direct sum of all the eigenspaces is dense in $\mathcal D(\mathcal M)$ in the topology of the uniform convergence
(and, a fortiori, in $L^2(\mathcal M,d\mu)$). The value $\lambda^2=0$ is an eigenvalue and the associated eigenspace is the $1$-dimensional space consisting of the constant functions on $\mathcal M$.
It is therefore convenient to list the eigenvalues as a non decreasing sequence
\[
\lambda_0^2=0<\lambda_1^2\le\lambda_2^2\le\ldots
\]
where the repetitions correspond to eigenvalues with multiplicity greater than $1$.
We can therefore associate each eigenvalue $\lambda_k^2$ with an eigenfunction $\varphi_k$
in such a way that $\{\varphi_k\}_{k=0}^{+\infty}$ forms an orthonormal basis of $L^2(\mathcal M,d\mu)$.

We conclude this section with a result concerning certain estimates on the derivatives of the heat kernel. See \cite[Theorem 1.4.3 and the following remarks]{Gr}, and \cite[Theorem 5.5]{Ko} for an extension to general elliptic operators on manifolds of bounded geometry. See also 
\cite[Chapter 3.E]{BGM} for details on the Minakshisundaram-Pleijel asymptotic development 
of the heat kernel.
\begin{theorem}\label{greiner}
Let $\mathcal M$ be a compact Riemannian manifold of dimension $d$.
Let $X_1,\ldots,\,X_\ell$ and $Y_1,\ldots,\,Y_m$ be differentiable vector fields on $\mathcal M$
such that $\|X_j\|_x\le 1$ and $\|Y_i\|_x\le 1$ for all $x\in\mathcal M$, for all $j=1,\ldots,\ell$ and
for all $i=1,\ldots,\,m$.
Then there exist two positive constants $c_7=c_7(\ell,m)$ and $c_8=c_8(\ell,m)$ depending only on 
$\ell,\,m$ (and on $\mathcal M$) such that for all $t\in[0,1]$ and for all $x,\,y\in \mathcal M$,
\[
\left|\sum_{k=0}^{+\infty}\exp(-\lambda_k^2t)X_1\ldots X_\ell\varphi_k(x)Y_1\ldots Y_m\varphi_k(y)\right|
\le
c_7t^{-\frac{d+\ell+m}2}\exp\left(-c_8\frac{|x-y|^2}{t}\right).
\]
\end{theorem}

\section{Estimates on summability kernels}
Here we recall certain definitions and results concerning general summability kernels
for Bessel systems, following \cite{FM2010}. 
\begin{definition}
A system $\{\phi_k\}_{k=0}^{+\infty}\subset L^2(\mathcal M)$ will be called a generalized Bessel system if for any $g\in\mathcal D(\mathcal M)$,
\[
\mathcal N(g):=\sum_{k=0}^{+\infty}\left|\int_{\mathcal M}g{\phi_k} d\mu\right|^2<+\infty.
\]
\end{definition}

Any Bessel system, that is a system $\{\phi_k\}_{k=0}^{+\infty}$ such that 
for all $f\in L^2({\mathcal M})$ one has $\sum_{k=0}^{+\infty}\left|\int_{\mathcal M}f{\phi_k} d\mu\right|^2\le c\|f\|_2^2$, is clearly a generalized Bessel system according to the above definition.
In particular, any orthonormal system is a generalized Bessel system.
We will also use the following type of generalized Bessel systems.
Let $X$ be a differentiable vector field on $\mathcal M$ and set $\phi_k=X\varphi_k$, where $\{\varphi_k\}_{k=0}^{+\infty}$ are the eigenfunctions of the
Laplace-Beltrami operator described before. 
By Green's formula (\ref{Green}),
for any $g\in\mathcal D(\mathcal M)$
\[
\sum_{k=0}^{+\infty}\left|\int_{\mathcal M}g ({X\varphi_k}) d\mu\right|^2
=\sum_{k=0}^{+\infty}\left|-\int_{\mathcal M}{\mathrm{div}( gX)} {\varphi_k} d\mu\right|^2=\|\mathrm{div}( gX)\|_2^2<+\infty.
\]  
Similarly, if $X_1,\, X_2$ are differentiable vector fields on $\mathcal M$ and we set $\phi_k=X_1X_2\varphi_k$, then
\[
\sum_{k=0}^{+\infty}\left|\int_{\mathcal M}g ({X_1X_2\varphi_k}) d\mu\right|^2
=\|\mathrm{div}( \mathrm{div}(gX_1)X_2)\|_2^2<+\infty,
\]
and so on for any number of vector fields.

Here is the main result of this section. 

\begin{theorem}[Theorem 2.1 in \cite{FM2010}] \label{kernel} Let $\{\phi_k\}$, $\{\psi_k\}$ be generalized Bessel systems composed by continuous functions,
and assume that there exist positive constants $\kappa_1,\,\kappa_2,$ $\kappa_3,\,\kappa_4,\,A_1,\,A_2,\,A_3$ such that for all $t\in(0,1]$ and $x,\,y\in\mathcal M$
\begin{align*}
\sum_{k=0}^{+\infty}\exp(-\lambda_k^2t)|\phi_k(x)|^2&\le\kappa_1 t^{-A_1/2},\\
\sum_{k=0}^{+\infty}\exp(-\lambda_k^2t)|\psi_k(x)|^2&\le\kappa_2 t^{-A_2/2},\\
\left|\sum_{k=0}^{+\infty}\exp(-\lambda_k^2t)\phi_k(x){\psi_k(y)}\right|&\le\kappa_3 t^{-A_3}\exp(-\kappa_4|x-y|^2/t).
\end{align*}
Let $K=(A_1+A_2)/2$ and $S>\max\{K,\,d\}$ integer. Let $H:\mathbb R\to \mathbb R$
be an even function supported on $[-1,\,1]$ with continuous derivatives up to order $S$. 
Then there exists a positive constant $C_9$ such that for all $x,\,y\in\mathcal M$ and for all $L>0$,
\[
\left|\sum_{k=0}^{+\infty}H(\lambda_k/L)\phi_k(x){\psi_k(y)}\right|\le C_9\frac{L^K}{(1+L|x-y|)^{S}}.
\]
\end{theorem}
The constant $C_9$ does not depend on the specific Bessel systems $\{\phi_k\}$ and $\{\psi_k\}$ nor on the manifold $\mathcal M$, but only on the constants $\kappa_1,\,\kappa_2,\,\kappa_3,\,\kappa_4,\,A_1,\,A_2,\,A_3$, and on the function $H$.

In particular, assume that the generalized Bessel systems are given by 
\[
\phi_k=X_1\ldots X_\ell\varphi_k,\quad\psi_k=Y_1\ldots Y_m\varphi_k
\]
where  $X_1,\ldots,\,X_\ell$ and $Y_1,\ldots,\,Y_m$ are differentiable vector fields on $\mathcal M$ such that $\|X_j\|_x\le 1$ and $\|Y_i\|_x\le 1$ for all $x\in\mathcal M$, for all $j=1,\ldots,\ell$ and
for all $i=1,\ldots,\,m$. Then, by Theorem \ref{greiner}, the above Theorem \ref{kernel} applies with $\kappa_1=c_7(\ell,\ell)$, $\kappa_2=c_7(m,m)$, $\kappa_3=c_7(\ell,m)$, $\kappa_4=c_8(\ell, m)$, $A_1=d+2\ell$, $A_2=d+2m$, and $A_3=(d+\ell+m)/2$.

Theorem \ref{kernel} follows
from a series of results, the first being Theorem 4.1 again in \cite{FM2010}.
The latter is a nice elementary result on holomorphic functions, which is a simplified
version of a result of A. Sikora \cite[Theorem 2]{Sikora}. Unfortunately, the
proof of Theorem 4.1 presented in \cite{FM2010} is incomplete, as inequality (4.12)
is not properly justified. Actually one needs to use, as A.
Sikora does in his original proof, some version of the Phragm\'en-Lindel\"of
theorem. We give here a complete proof of this result.

\begin{theorem}[Theorem 4.1 in \cite{FM2010}]\label{T41_FM}
Let $r>0$, $\{a_{k}\}$ be an absolutely summable sequence of complex numbers,
$\{\ell_{k}\}$ be a sequence of nonnegative, nondecreasing numbers with
$\ell_{k}\rightarrow\infty$, and
\[
K\left(  t\right)  =\sum_{k=0}^{\infty}\exp\left(  -\ell_{k}^{2}t\right)
a_{k},\quad W\left(  t\right)  =\sum_{k=0}^{\infty}\cos\left(  \ell
_{k}t\right)  a_{k}.
\]
Then%
\[
\left\vert K\left(  t\right)  \right\vert \leq \alpha t^{-\beta}\exp\left(
-r^{2}/t\right)  \sum_{k=0}^{\infty}\left\vert a_{k}\right\vert ,\quad
t\in\left(  0,1\right]
\]
if and only if $W\left(  t\right)  =0$ for $0\leq t\leq2r$.
\end{theorem}

\begin{proof}
Assume without loss of generality that $\sum_{k=0}^{\infty}\left\vert
a_{k}\right\vert =1.$ By the well known formula on the Fourier transform of the Gaussian function, for all $t>0$ and for all
$k=0,1,2,\ldots$%
\[
\exp\left(  -\ell_{k}^{2}t\right)  =\frac{1}{\sqrt{\pi t}}\int_{0}^{+\infty
}\exp\left(  -u^{2}/\left(  4t\right)  \right)  \cos\left(  \ell_{k}u\right)
du
\]
and, by the absolute and uniform convergence of the series defining $W\left(
t\right)  $,%
\begin{equation}
K\left(  t\right)  =\frac{1}{\sqrt{\pi t}}\int_{0}^{+\infty}\exp\left(
-u^{2}/\left(  4t\right)  \right)  W\left(  u\right)  du \label{key}%
\end{equation}
Assume now $W\left(  u\right)  =0$ for all $0\leq u\leq2r$. If $t\geq r^{2}$
then%
\[
\left\vert K\left(  t\right)  \right\vert \leq1\leq\exp\left(  1-r^{2}%
/t\right)  =e\exp\left(  -r^{2}/t\right)  .
\]
If $t<r^{2}$ then, by (\ref{key}),%
\begin{align*}
\left\vert K\left(  t\right)  \right\vert  &  \leq\frac{1}{\sqrt{\pi t}}%
\int_{2r}^{+\infty}\exp\left(  -u^{2}/\left(  4t\right)  \right)  \left\vert
W\left(  u\right)  \right\vert du\\
&  \leq\frac{1}{\sqrt{\pi t}}\int_{2r}^{+\infty}\exp\left(  -u^{2}/\left(
4t\right)  \right)  du\\
&  =\frac{1}{\sqrt{\pi}}\int_{r^{2}/t}^{+\infty}u^{-1/2}\exp\left(  -u\right)
du\\
&  \leq\frac{1}{\sqrt{\pi}}\left(  \frac{t}{r^{2}}\right)  ^{1/2}\exp\left(
-\frac{r^{2}}{t}\right)  \leq\frac{1}{\sqrt{\pi}}\exp\left(  -\frac{r^{2}}%
{t}\right)  ,
\end{align*}
and the thesis follows with $\alpha=e$ and $\beta=0$.

Assume now that
\[
\left\vert K\left(  t\right)  \right\vert \leq \alpha t^{-\beta}\exp\left(
-r^{2}/t\right)  ,\quad t\in\left(  0,1\right]  .
\]
Then for any $\varepsilon>0$ there exists a positive constant $\gamma$ such
that
\begin{equation}
\left\vert K\left(  t\right)  \right\vert \leq \gamma\exp\left(  -\left(
r-\varepsilon\right)  ^{2}/t\right)  ,\quad t\in\left(  0,1\right]  .
\label{hyp}%
\end{equation}
For any complex $t=\tau+i\xi$ with $\tau\geq0,$ define the function
\[
F\left(  t\right)  =\left\{
\begin{array}
[c]{ll}%
\dfrac{t}{1+t}\exp\left(  4\left(  r-\varepsilon\right)  ^{2}t\right)
{\displaystyle\sum\limits_{k=0}^{+\infty}}
\exp\left(  -\ell_{k}^{2}/\left(  4t\right)  \right)  a_{k} & t\neq
0,\,\tau\geq0\\
0 & t=0.
\end{array}
\right.
\]
It is easy to show that $F$ satisfies the hypotheses of the Phragm\'en-Lindel\"of
theorem (see \cite[Theorem 3.4 page 124]{Stein}) in $\left\{  -\frac{\pi}{2}<\arg\left(  t\right)  <0\right\}  $ and in
$\left\{  0<\arg\left(  t\right)  <\frac{\pi}{2}\right\}  .$ In particular, by
(\ref{hyp}), for all $\tau>0$%
\begin{align*}
\left\vert F\left(  \tau\right)  \right\vert  &  \leq\left\{
\begin{array}
[c]{ll}%
\dfrac{\tau}{1+\tau}\exp\left(  4\left(  r-\varepsilon\right)  ^{2}%
\tau\right)  \gamma\exp\left(  -\left(  r-\varepsilon\right)  ^{2}4\tau\right)
& \text{if }1/\left(  4\tau\right)  \leq1\\
\\
\dfrac{\tau}{1+\tau}\exp\left(  4\left(  r-\varepsilon\right)  ^{2}\tau\right)
& \text{if }1/\left(  4\tau\right)  \geq1
\end{array}
\right. \\
&  \leq\left\{
\begin{array}
[c]{ll}%
\gamma & \text{if }\tau\geq1/4\\
\\
\dfrac{1}{5}\exp\left(  \left(  r-\varepsilon\right)  ^{2}\right)  & \text{if
}\tau\leq1/4
\end{array}
\right. \\
&  \leq\max\left\{  \gamma,\dfrac{1}{5}\exp\left(  r^{2}\right),1 \right\}
=:\eta.
\end{align*}
Also, for all $\xi\in\mathbb{R\setminus}\left\{  0\right\}  $,%
\[
\left\vert F\left(  i\xi\right)  \right\vert \leq\dfrac{\left\vert
\xi\right\vert }{\left\vert 1+i\xi\right\vert }%
{\displaystyle\sum\limits_{k=0}^{+\infty}}
\left\vert a_{k}\right\vert \leq1\le\eta.
\]
The function $F$ is continuous on $\left\{  \operatorname{Re}\left(  t\right)
\geq0\right\}  .$ The continuity outside $0$ follows by the absolute and
uniform convergence of the series defining $F$ in any compact subset of
$\left\{  \operatorname{Re}\left(  t\right)  \geq0\right\}  \setminus\left\{
0\right\}  $. The continuity in $0$ follows by the estimate, for
$\operatorname{Re}\left(  t\right)  \geq0$ and $0<\left\vert t\right\vert
<\delta$,
\begin{align*}
\left\vert F\left(  t\right)  \right\vert  &  \leq\left\vert t\right\vert
\exp\left(  4\left(  r-\varepsilon\right)  ^{2}\delta\right)
{\displaystyle\sum\limits_{k=0}^{+\infty}}
\exp\left(  -\ell_{k}^{2}\tau/\left(  4\left\vert t\right\vert ^{2}\right)
\right)  \left\vert a_{k}\right\vert \\
&  \leq\left\vert t\right\vert \exp\left(  4\left(  r-\varepsilon\right)
^{2}\delta\right)  .
\end{align*}
Again by the absolute and uniform convergence of the series defining $F$, it
follows that $F$ is holomorphic in $\left\{  \operatorname{Re}\left(
t\right)  >0\right\}  .$ Finally, for all $\operatorname{Re}\left(  t\right)
>0$,
\begin{align*}
\left\vert F\left(  t\right)  \right\vert  &  \leq\exp\left(  4\left(
r-\varepsilon\right)  ^{2}\tau\right)
{\displaystyle\sum\limits_{k=0}^{+\infty}}
\exp\left(  -\ell_{k}^{2}\tau/\left(  4\left(  \tau^{2}+\xi^{2}\right)  \right)
\right)  \left\vert a_{k}\right\vert \\
&  \leq\exp\left(  4\left(  r-\varepsilon\right)  ^{2}\tau\right)  \leq
\exp\left(  4\left(  r-\varepsilon\right)  ^{2}\left\vert t\right\vert
\right)  .
\end{align*}
It therefore follows by the Phragm\'en-Lindel\"of theorem that
\[
\left\vert F\left(  t\right)  \right\vert \leq \eta,\quad\forall
t:\operatorname{Re}\left(  t\right)  \geq0.
\]
The proof now follows as in \cite{FM2010}. Changing variables in (\ref{key})
we obtain
\[
\frac{1}{\sqrt{4\tau}}K\left(  \frac{1}{4\tau}\right)  =\int_{-\infty
}^{+\infty}\exp\left(  -u\tau\right)  g\left(  u\right)  du
\]
where
\[
g\left(  u\right)  =\left\{
\begin{array}
[c]{ll}%
\dfrac{1}{\sqrt{4\pi u}}W\left(  \sqrt{u}\right)  & \text{if }u>0\\
0 & \text{if }u\leq0.
\end{array}
\right.
\]
By the definition of $F$ we have for $\tau>0$%
\[
\frac{1}{\sqrt{4\tau}}K\left(  \frac{1}{4\tau}\right)  =\frac{1}{\sqrt{4\tau}%
}\frac{1+\tau}{\tau}\exp\left(  -4\left(  r-\varepsilon\right)  ^{2}%
\tau\right)  F\left(  \tau\right)  ,
\]
which can be extended analytically to $\left\{  \operatorname{Re}\left(
t\right)  >0\right\}  $. Also $\int_{-\infty}^{+\infty}\exp\left(
-u\tau\right)  g\left(  u\right)  du$ can be extended analytically to
$\left\{  \operatorname{Re}\left(  t\right)  >0\right\}  $, and the identity%
\begin{equation}
\label{F(t)}
\frac{1}{\left(  4t\right)  ^{1/2}}\frac{1+t}{t}\exp\left(  -4\left(
r-\varepsilon\right)  ^{2}t\right)  F\left(  t\right)  =\int_{-\infty
}^{+\infty}\exp\left(  -ut\right)  g\left(  u\right)  du
\end{equation}
holds in $\left\{  \operatorname{Re}\left(  t\right)  >0\right\}  $. Let now
$\phi\in\mathcal{C}^{\infty}\left(  \mathbb{R}\right)  $ with support in
$\left[  0,b\right]  .$ Then%
\[
\widehat{\phi}\left(  t\right)  =\int_{-\infty}^{+\infty}\phi\left(  u\right)
e^{-iut}du
\]
is entire and a repeated integration by parts gives
\begin{align*}
\widehat{\phi}\left(  -\xi+i\tau\right)   &  =\int_{-\infty}^{+\infty}%
\phi\left(  u\right)  e^{-iu\left(  -\xi+i\tau\right)  }du\\
&  =\frac{\left(  -1\right)  ^{R}}{\left(  -i\left(  -\xi+i\tau\right)
\right)  ^{R}}\int_{-\infty}^{+\infty}\phi^{\left(  R\right)  }\left(
u\right)  e^{-iu\left(  -\xi+i\tau\right)  }du\\
&  =\frac{\left(  -1\right)  ^{R}}{\left(  \tau+i\xi\right)  ^{R}}%
\int_{-\infty}^{+\infty}\phi^{\left(  R\right)  }\left(  u\right)  e^{\left(
\tau+i\xi\right)  u}du,
\end{align*}
so that
\begin{equation}\label{phihat}
\left\vert \widehat{\phi}\left(  -\xi+i\tau\right)  \right\vert \leq\frac
{\max\left(  1,e^{\tau b}\right)  }{\left(  \tau^{2}+\xi^{2}\right)  ^{R/2}%
}\int_{-\infty}^{+\infty}\left\vert \phi^{\left(  R\right)  }\left(  u\right)
\right\vert du.
\end{equation}
Thus, for any $\tau>0$,
\begin{align*}
&  \int_{\mathbb{R}}g\left(  u\right)  \phi\left(  u\right)  du\\
&  =\int_{\mathbb{R}}e^{-u\tau}g\left(  u\right)  \left\{  e^{u\tau}%
\phi\left(  u\right)  \right\}  du\\
&  =\int_{\mathbb{R}}e^{-u\tau}g\left(  u\right)  \left\{  \frac{1}{2\pi}%
\int_{\mathbb{R}}e^{i\xi u}\widehat{\left(  e^{\cdot\tau}\phi\left(
\cdot\right)  \right)  }\left(  \xi\right)  d\xi\right\}  du\\
&  =\int_{\mathbb{R}}e^{-u\tau}g\left(  u\right)  \frac{1}{2\pi}%
\int_{\mathbb{R}}e^{i\xi u}\int_{\mathbb{R}}e^{s\tau}\phi\left(  s\right)
e^{-is\xi}dsd\xi du\\
&  =\int_{\mathbb{R}}e^{-u\tau}g\left(  u\right)  \frac{1}{2\pi}%
\int_{\mathbb{R}}e^{i\xi u}\int_{\mathbb{R}}\phi\left(  s\right)
e^{-is\left(  \xi+i\tau\right)  }dsd\xi du\\
&  =\int_{\mathbb{R}}e^{-u\tau}g\left(  u\right)  \frac{1}{2\pi}%
\int_{\mathbb{R}}e^{-i\xi u}\int_{\mathbb{R}}\phi\left(  s\right)
e^{-is\left(  -\xi+i\tau\right)  }dsd\xi du\\
&  =\int_{\mathbb{R}}e^{-u\tau}g\left(  u\right)  \frac{1}{2\pi}%
\int_{\mathbb{R}}e^{-i\xi u}\widehat{\phi}\left(  -\xi+i\tau\right)  d\xi du.
\end{align*}
One can apply Fubini's theorem, since%
\[
\left\vert e^{-u\tau}g\left(  u\right)  e^{-i\xi u}\widehat{\phi}\left(
-\xi+i\tau\right)  \right\vert \leq e^{-u\tau}\left\vert g\left(  u\right)
\right\vert \frac{e^{\tau b}}{\left(  \tau^{2}+\xi^{2}\right)  ^{R/2}}%
\int_{-\infty}^{+\infty}\left\vert \phi^{\left(  R\right)  }\left(  u\right)
\right\vert du,
\]
so that by (\ref{F(t)})
\begin{align*}
&\int_{\mathbb{R}}g\left(  u\right)  \phi\left(  u\right)  du  \\  &=\frac
{1}{2\pi}\int_{\mathbb{R}}\widehat{\phi}\left(  -\xi+i\tau\right)
\int_{\mathbb{R}}g\left(  u\right)  e^{-u\left(  \tau+i\xi\right)  }dud\xi\\
 &=\frac{1}{2\pi}\int_{\mathbb{R}}\widehat{\phi}\left(  -\xi+i\tau\right)
\frac{1}{\left(  4\left(  \tau+i\xi\right)  \right)  ^{1/2}}\frac{1+\tau+i\xi
}{\tau+i\xi}e^{  -4\left(  r-\varepsilon\right)  ^{2}\left(  \tau
+i\xi\right) }  F\left(  \tau+i\xi\right)  d\xi.
\end{align*}
It follows from this and \eqref{phihat} that
\begin{align*}
&  \left\vert \int_{\mathbb{R}}g\left(  u\right)  \phi\left(  u\right)
du\right\vert \\
&  \leq\frac{\eta}{4\pi}e^{ -4\left(  r-\varepsilon
\right)  ^{2}\tau}\int_{\mathbb{R}}\left\vert \widehat{\phi}\left(
-\xi+i\tau\right)  \right\vert \frac{\left\vert 1+\tau+i\xi\right\vert
}{\left\vert \tau+i\xi\right\vert ^{3/2}}  d\xi\\
&  \leq\frac{\eta}{4\pi}e^ { \left(  b-4\left(
r-\varepsilon\right)  ^{2}\right)  \tau}\int_{-\infty}^{+\infty}\left\vert \phi^{\left(
R\right)  }\left(  u\right)  \right\vert du\int_{\mathbb{R}}\left(  \left(
1+\tau\right)  ^{2}+\xi^{2}\right)  ^{1/2}\left(  \tau^{2}+\xi^{2}\right)
^{- R/2  - 3/4  } d\xi
\end{align*}
which goes to $0$ as $\tau\rightarrow+\infty$ if $b\leq4\left(  r-\varepsilon
\right)  ^{2}$. By the arbitrarity of $\phi$ it follows that $g\left(
u\right)  =0$ for $0\leq u\leq4\left(  r-\varepsilon\right)  ^{2}$, and by the
arbitrarity of $\varepsilon>0,$ $g\left(  u\right)  =0$ for $0\leq u\leq
4r^{2}$, so that $W\left(  t\right)  =0$ for $0\leq t\leq2r$.
\end{proof}

The next step in the proof of Theorem \ref{kernel}, again following \cite{FM2010},
is

\begin{corollary}[Corollary 4.1 in \cite{FM2010}]
\label{cor_10}
Let $G:\mathbb R\to \mathbb R$ be an even, bounded, integrable function such that 
the Fourier transform
$
\widehat G
$
is also integrable and supported on $[-2r,\,2r]$. 
In the assumptions of Theorem \ref{T41_FM}, if 
\[
\left\vert K\left(  t\right)  \right\vert \leq \alpha t^{-\beta}\exp\left(
-r^{2}/t\right)  \sum_{k=0}^{\infty}\left\vert a_{k}\right\vert ,\quad
t\in\left(  0,1\right]
\]
then 
\[
\sum_{k=0}^{+\infty} G(\ell_k)a_k=0.
\]
\end{corollary} 

\begin{proof}
By the Fourier inversion formula,
\[
G(u)=\frac 1\pi\int_0^{+\infty}\widehat G(t)\cos(tu)dt.
\]
Thus,
\[
\sum_{k=0}^{+\infty}G(\ell_k)a_k=\frac 1\pi\int_0^{+\infty}\widehat G(t)W(t)dt=0
\]
by Theorem \ref{T41_FM}.
\end{proof}

Let now $V:\mathbb R\to\mathbb R$ be an even function such that $\widehat V$ is infinitely differentiable with $\widehat V(u)=1 $ when $|u|\le1/2$ and $\widehat V(u)=0 $ when $|u|\ge1$. Then for any $Y>0$ let $H_Y$ be defined by
\[
\widehat H_Y(u)=\widehat H(u)\widehat V(u/Y).
\]

\begin{lemma}[Lemma 4.3 in \cite{FM2010}, Lemma 6.1 in \cite{MM}]
\label{nice_kernel}
In the assumptions of Theorem~\ref{kernel}, there is a constant $c_{10}>0$ (depending on $\kappa_1,\,\kappa_2,\,\kappa_3,\,\kappa_4,\,A_1,\,A_2,\,A_3$ and on the function $H$) such that for all $Y\ge1/2$ and for all $x,y\in\mathcal M$
\[
\left|
\sum_{k=0}^{+\infty}\left(H(\lambda_k/L)-H_Y(\lambda_k/L)\right)\phi_k(x){\psi_k(y)}
\right|
\le c_{10}L^KY^{-S}.
\]
\end{lemma}
The interested reader can find the proof in the above mentioned references.

\begin{proof}[Proof of Theorem \ref{kernel}]
Again, we follow Filbir and Mhaskar \cite[Proof of Theorem~2.1 page 646]{FM2010}.
By the hypotheses, for all $L\ge 1$ we have
\[
\sum_{\lambda_k\le L}|\phi_k(x)|^2\le\sum_{k=0}^{+\infty}\exp(1-\lambda_k^2/L^2)|\phi_k(x)|^2 \le e\kappa_1 L^{A_1},
\]
and similarly,
\[
\sum_{\lambda_k\le L}|\psi_k(x)|^2\le e\kappa_2 L^{A_2}.
\]
Hence, by the Cauchy-Schwarz inequality
\begin{align}
\nonumber\left|
\sum_{k=0}^{+\infty}H(\lambda_k/L)\phi_k(x){\psi_k(y)}
\right|&\le \max_{t\in\mathbb R}|H(t)|\left(\sum_{\lambda_k\le L}|\phi_k(x)|^2\right)^{1/2}
\left(\sum_{\lambda_k\le L}|\psi_k(y)|^2\right)^{1/2}\\
&\le e(\kappa_1\kappa_2)^{1/2}\max_{t\in\mathbb R}|H(t)| L^K \label {f1}.
\end{align}
This proves the theorem when $L|x-y|\le1$. Assume now $L|x-y|\ge1$ and let $Y:=\sqrt{\kappa_4}L|x-y|$.
Let $f_1,\,f_2\in L^1(\mathcal M)$ with $\|f_1\|_1=\|f_2\|_1=1$ be supported in the balls centered
at $x$ and $y$ respectively, and with radii $|x-y|/8$. For any $\varepsilon >0$ there exist two functions 
$g_1,\,g_2\in\mathcal C^{\infty}(\mathcal M)$ supported in the balls centered
at $x$ and $y$ respectively, and with radii $|x-y|/4$ such that 
$\|f_1-g_1\|_1<\varepsilon$ and $\|f_2-g_2\|_1<\varepsilon$.
Therefore, by (\ref{f1}) and Lemma \ref{nice_kernel},
\begin{align*}
&
\left|\sum_{k=0}^{+\infty}H(\lambda_k/L)\int_{\mathcal M}\int_{\mathcal M}\phi_k(w){\psi_k(z)}
{f_1(w)} {f_2(z)}d\mu(w)d\mu(z)
\right|\\
&\le 
\left|\sum_{k=0}^{+\infty}H(\lambda_k/L)\int_{\mathcal M}\int_{\mathcal M}\phi_k(w){\psi_k(z)}
{f_1(w)} ({f_2(z)}-g_2(z))d\mu(w)d\mu(z)\right|\\
&+\left|\sum_{k=0}^{+\infty}H(\lambda_k/L)\int_{\mathcal M}\int_{\mathcal M}\phi_k(w){\psi_k(z)}
({f_1(w)-g_1(w)}){g_2(z)}d\mu(w)d\mu(z)\right|\\
&+\left|\sum_{k=0}^{+\infty}H(\lambda_k/L)\int_{\mathcal M}\int_{\mathcal M}\phi_k(w){\psi_k(z)}
{g_1(w)}{g_2(z)}d\mu(w)d\mu(z)\right|\\
&\le\left|\sum_{k=0}^{+\infty}H(\lambda_k/L)\int_{\mathcal M}\int_{\mathcal M}\phi_k(w){\psi_k(z)}
{g_1(w)}{g_2(z)}d\mu(w)d\mu(z)\right|
\\
&+e(\kappa_1\kappa_2)^{1/2}\varepsilon  \max_{t\in\mathbb R}|H(t)|(\|f_1\|_1+\|g_2\|_1)L^K\\
&\le \left|\sum_{k=0}^{+\infty}H_Y(\lambda_k/L)\int_{\mathcal M}\int_{\mathcal M}\phi_k(w){\psi_k(z)}
{g_1(w)} {g_2(z)}d\mu(w)d\mu(z)
\right|\\
&+c_{10}L^K Y^{-S}\|g_1\|_1\|g_2\|_1+e(\kappa_1\kappa_2)^{1/2}\varepsilon  \max_{t\in\mathbb R}|H(t)|(\|f_1\|_1+\|g_2\|_1)L^K.
\end{align*}
The distance between the supports of $g_1$ and $g_2$ exceeds $|x-y|/2$ and therefore
for all $t\in(0,\,1]$
\begin{align}
\nonumber
&\left| \sum_{k=0}^{+\infty}\exp(-\lambda_k^2t)\int_{\mathcal M}\int_{\mathcal M}\phi_k(w){\psi_k(z)}
{g_1(w)}{g_2(z)}d\mu(w)d\mu(z)\right|\\
\label{ipotesi}
&\le\kappa_3t^{-A_3}\exp(-\kappa_4|x-y|^2/(4t))\|g_1\|_1\|g_2\|_1.
\end{align}
Observe that since $\{\phi_k\}$ and $\{\psi_k\}$ are generalized Bessel systems, 
then 
\[
\sum_{k=0}^{+\infty}|\int_{\mathcal M}\int_{\mathcal M}\phi_k(w){\psi_k(z)}
{g_1(w)}{g_2(z)}d\mu(w)d\mu(z)|
\le\left(\mathcal N_{\{\phi_k\}}(g_1){\mathcal N}_{\{\psi_k\}}(g_2)\right)^{1/2}.
\]
We may therefore apply Corollary \ref{cor_10} with $r=\sqrt{\kappa_4}|x-y|/2$, $G(u)=H_Y(u/L)$ and
\[
a_k=\int_{\mathcal M}\int_{\mathcal M}\phi_k(w){\psi_k(z)}
{g_1(w)}{g_2(z)}d\mu(w)d\mu(z).
\]
By (\ref{ipotesi}),
\[
\sum_{k=0}^{+\infty}H_Y(\lambda_k/L)\int_{\mathcal M}\int_{\mathcal M}\phi_k(w){\psi_k(z)}
{g_1(w)}{g_2(z)}d\mu(w)d\mu(z)=0.
\]
Therefore
\begin{align*}
&\left|\sum_{k=0}^{+\infty}H(\lambda_k/L)\int_{\mathcal M}\int_{\mathcal M}\phi_k(w){\psi_k(z)}
{f_1(w)} {f_2(z)}d\mu(w)d\mu(z)
\right|\\
&\le
c_{10}L^K Y^{-S}\|g_1\|_1\|g_2\|_1+e(\kappa_1\kappa_2)^{1/2}\varepsilon  \max_{t\in\mathbb R}|H(t)|(\|f_1\|_1+\|g_2\|_1)L^K\\
&\le c_{10}(1+\varepsilon)^2L^K Y^{-S}+e(\kappa_1\kappa_2)^{1/2}\varepsilon (2+\varepsilon) \max_{t\in\mathbb R}|H(t)|L^K.
\end{align*}
By the arbitrarity of $\varepsilon >0$ this gives
\[
\left|\sum_{k=0}^{+\infty}H(\lambda_k/L)\int_{\mathcal M}\int_{\mathcal M}\phi_k(w){\psi_k(z)}
{f_1(w)} {f_2(z)}d\mu(w)d\mu(z)
\right|
\le c_{10}L^K Y^{-S},
\]
and by the arbitrarity of $f_1$ and $f_2$ and the continuity of $\phi_k$ and $\psi_k$,
\[
\left|\sum_{k=0}^{+\infty}H(\lambda_k/L)\phi_k(x){\psi_k(y)}
\right|
\le C_{9}L^K Y^{-S}.
\]
\end{proof}

\section{Proof of Theorem \ref{MZ gradients}}

This proof follows the lines of the corresponding proof of Theorem \ref{MZ polys}
by Filbir and Mhaskar as found in \cite{FM2010, FM2011}, properly modified
to treat the case of the gradients.
 
Fix $\varepsilon>0,$ and let $v_{\varepsilon}:\mathbb{[}0,\mathbb{+\infty
)}\rightarrow\mathbb{R}$ be a $\mathcal C^{\infty}$ function such that $v_{\varepsilon
}\left(  u\right)  =u$ for $u\geq\varepsilon$, $v_{\varepsilon}\left(
u\right)  =\varepsilon/2$ for $u<\varepsilon/4$ and $v_{\varepsilon}\left(
u\right)  \geq u$ for all $u\geq0$. Let $P\in\Pi_{L}^{0}$ and define the
differentiable vector field
\[
T\left(  x\right)  :=\frac{\nabla P\left(  x\right)  }{v_{\varepsilon}\left(
\left\Vert \nabla P\left(  x\right)  \right\Vert \right)  }.
\]
Then
\[
TP\left(  x\right)  =\left\langle \nabla P\left(  x\right)  ,\frac{\nabla
P\left(  x\right)  }{v_{\varepsilon}\left(  \left\Vert \nabla P\left(
x\right)  \right\Vert \right)  }\right\rangle =\frac{\left\Vert \nabla
P\left(  x\right)  \right\Vert ^{2}}{v_{\varepsilon}\left(  \left\Vert \nabla
P\left(  x\right)  \right\Vert \right)  }\leq\left\Vert \nabla P\left(
x\right)  \right\Vert
\]
and%
\begin{align*}
&  \left\vert \int_{\mathcal{M}}\left\Vert \nabla P\left(  x\right)
\right\Vert d\mu\left(  x\right)  -\sum_{j=1}^{N}\frac{1}{N}\left\Vert \nabla
P\left(  x_{j}\right)  \right\Vert \right\vert \\
&  \leq\left\vert \int_{\mathcal{M}}\left(  \left\Vert \nabla P\left(
x\right)  \right\Vert -TP\left(  x\right)  \right)  d\mu\left(  x\right)
\right\vert +\left\vert \int_{\mathcal{M}}TP\left(  x\right)  d\mu\left(
x\right)  -\sum_{j=1}^{N}\frac{1}{N}TP\left(  x_{j}\right)  \right\vert \\
&  +\left\vert \sum_{j=1}^{N}\frac{1}{N}\left(  TP\left(  x_{j}\right)
-\left\Vert \nabla P\left(  x_{j}\right)  \right\Vert \right)  \right\vert \\
&  \leq2\varepsilon+\left\vert \int_{\mathcal{M}}TP\left(  x\right)
d\mu\left(  x\right)  -\sum_{j=1}^{N}\frac{1}{N}TP\left(  x_{j}\right)
\right\vert .
\end{align*}
Let us now call $\delta$ the maximum diameter of the balls $X_{j}$, so that  $\delta\le2c_2N^{-1/d}.$ Now%
\begin{align*}
\left\vert \int_{\mathcal{M}}TP\left(  x\right)  d\mu\left(  x\right)
-\sum_{j=1}^{N}\frac{1}{N}TP\left(  x_{j}\right)  \right\vert  &  \leq
\sum_{j=1}^{N}\int_{R_{j}}\left\vert TP\left(  x\right)  -TP\left(
x_{j}\right)  \right\vert d\mu\left(  x\right)  \\
&  \leq\sum_{j=1}^{N}\frac{1}{N}\sup_{x,z\in R_{j}}\left\vert TP\left(
x\right)  -TP\left(  z\right)  \right\vert 
\end{align*}
By (\ref{geodesic}), the last term can be bounded above by 
\[
\sum_{j=1}^{N}\frac{1}{N}\sup_{x,z\in R_j}\sup_{t\in [0,|x-y|]}\left\Vert \nabla
TP\left(  \alpha(t)\right)  \right\Vert |x-z|.
\]
where $\alpha$ is a normalized geodesic joining $x$ and $z$. 
Since $R_j$ is contained in the ball $X_j$, the geodesic $\alpha$ is contained in the
ball $2X_j$ with the same center as $X_j$ and radius twice the radius of $X_j$.
It follows that the last term is bounded above by 
\[
\delta\sum_{j=1}^{N}\frac{1}{N}\sup_{x\in 2X_{j}}\left\Vert \nabla
TP\left(  x\right)  \right\Vert .
\]

Defining now the vector field%
\[
S\left(  x\right)  :=\frac{\nabla TP\left(  x\right)  }{v_{\varepsilon}\left(
\left\Vert \nabla TP\left(  x\right)  \right\Vert \right)  }%
\]
we have as before%
\[
STP\left(  x\right)=\frac{\|\nabla TP(x)\|^2}{v_{\varepsilon}(\|\nabla TP(x)\|)} \leq\left\Vert \nabla TP\left(  x\right)  \right\Vert
\]
and therefore%
\begin{align*}
&\delta\sum_{j=1}^{N}\frac{1}{N}\sup_{x\in 2X_{j}}\left\Vert \nabla TP\left(
x\right)  \right\Vert \\
 &  \leq\delta\sum_{j=1}^{N}\frac{1}{N}\sup_{x\in 2X_{j}%
}\left\vert \left\Vert \nabla TP\left(  x\right)  \right\Vert -STP\left(
x\right)  \right\vert +\delta\sum_{j=1}^{N}\frac{1}{N}\sup_{x\in 2X_{j}%
}\left\vert STP\left(  x\right)  \right\vert \\
&  \leq\delta\varepsilon+\delta\sum_{j=1}^{N}\frac{1}{N}\sup_{x\in 2X_{j}%
}\left\vert STP\left(  x\right)  \right\vert .
\end{align*}
So far we have obtained the inequality
\[
\left\vert \int_{\mathcal{M}}\left\Vert \nabla P\left(  x\right)  \right\Vert
d\mu\left(  x\right)  -\sum_{j=1}^{N}\frac{1}{N}\left\Vert \nabla P\left(
x_{j}\right)  \right\Vert \right\vert \leq\left(  2+\delta\right)
\varepsilon+\delta\sum_{j=1}^{N}\frac{1}{N}\sup_{x\in 2X_{j}}\left\vert
STP\left(  x\right)  \right\vert .
\]
Let $h$ be a $\mathcal C^{\infty}$ even function such that $h(u)$ equals $1$ for $u\in\left[
-1,1\right]  $ and $h(u) $ equals $0$ for $\left\vert u\right\vert \geq2$. For
any $L\geq0$, define the kernels%
\begin{align*}
W_{L}\left(  x,y\right)   &  =\sum_{0<\lambda_{k}}\frac{1}{\lambda
_{k}^{2}}h\left(  \frac{\lambda_{k}}{L}\right)  \varphi_{k}\left(  x\right)
\varphi_{k}\left(  y\right)  \\
\Psi_{L}\left(  x,y\right)   &  =\Delta_{y}W_{L}\left(  x,y\right)
=\sum_{0<\lambda_{k}}h\left(  \frac{\lambda_{k}}{L}\right)  \varphi
_{k}\left(  x\right)  \varphi_{k}\left(  y\right)  .
\end{align*}
Since $\Psi_{L}\left(  x,y\right)  $ is a reproducing kernel for $\Pi_{L}^{0}%
$, we have by Green's formula (\ref{Green})%
\begin{align*}
P\left(  x\right)   &  =\int_{\mathcal{M}}P\left(  y\right)  \Psi_{L}\left(
x,y\right)  d\mu\left(  y\right)  =\int_{\mathcal{M}}P\left(  y\right)
\Delta_{y}W_{L}\left(  x,y\right)  d\mu\left(  y\right)  \\
&  =\int_{\mathcal{M}}\left\langle \nabla_{y}P\left(  y\right)  ,\nabla
_{y}W_{L}\left(  x,y\right)  \right\rangle d\mu\left(  y\right)  .
\end{align*}
Thus%
\[
STP\left(  x\right)  =\int_{\mathcal{M}}\left\langle \nabla_{y}P\left(
y\right)  ,\nabla_{y}S_{x}T_{x}W_{L}\left(  x,y\right)  \right\rangle
d\mu\left(  y\right)
\]
and%
\begin{align*}
\left\vert STP\left(  x\right)  \right\vert  &  =\left\vert \int%
_{\mathcal{M}}\left\langle \nabla_{y}P\left(  y\right)  ,\nabla_{y}S_{x}%
T_{x}W_{L}\left(  x,y\right)  \right\rangle d\mu\left(  y\right)  \right\vert
\\
&  \leq\int_{\mathcal{M}}\left\Vert \nabla_{y}P\left(  y\right)  \right\Vert
\left\Vert \nabla_{y}S_{x}T_{x}W_{L}\left(  x,y\right)  \right\Vert
d\mu\left(  y\right)  .
\end{align*}
We will show in a moment that%
\begin{equation}
\left\Vert \nabla_{y}S_{x}T_{x}W_{L}\left(  x,y\right)  \right\Vert \leq\kappa
L^{d+1}\left(  1+L\left\vert x-y\right\vert \right)  ^{-d-1}%
.\label{kernel estimate}%
\end{equation}
This inequality implies that%
\begin{align*}
&  \delta\sum_{j=1}^{N}\frac{1}{N}\sup_{x\in 2X_{j}}\left\vert STP\left(
x\right)  \right\vert \\
&  \leq\delta\sum_{j=1}^{N}\frac{1}{N}\sup_{x\in 2X_{j}}\int_{\mathcal{M}%
}\left\Vert \nabla_{y}P\left(  y\right)  \right\Vert \left\Vert \nabla
_{y}S_{x}T_{x}W_{L}\left(  x,y\right)  \right\Vert d\mu\left(  y\right)  \\
&  \leq\kappa\delta\sum_{j=1}^{N}\frac{1}{N}\sup_{x\in 2X_{j}}\int%
_{\mathcal{M}}\left\Vert \nabla_{y}P\left(  y\right)  \right\Vert
L^{d+1}\left(  1+L\left\vert x-y\right\vert \right)  ^{-d-1}d\mu\left(
y\right)  \\
&  \leq\kappa\delta\int_{\mathcal{M}}\left\Vert \nabla_{y}P\left(  y\right)
\right\Vert \left\{  \sum_{j=1}^{N}\frac{L^{d+1}}{N}\sup_{x\in 2X_{j}}\left(
1+L\left\vert x-y\right\vert \right)  ^{-d-1}\right\}  d\mu\left(  y\right)  .
\end{align*}
For any fixed $y$, let now $J=\left\{  j:\mathrm{dist}\left(  2X_{j},y\right)
\geq2\delta\right\}  $ and $J^{\prime}$ its complement. It is easy to see that,
calling $q_{j}$ the point in $2X_{j}$ that is closest to $y$, and $p_{j}$ the
point in $2X_{j}$ that is farthest from $y,$ then if $j\in J$%
\[
1+\frac{L}{2}\left\vert p_{j}-y\right\vert \leq1+\frac{L}{2}\left(  \left\vert
q_{j}-y\right\vert +2\delta\right)  \leq1+L\left\vert q_{j}-y\right\vert
\]
and therefore by (\ref{radiale})
\begin{align*}
&  \sum_{j\in J}\frac{L^{d+1}}{N}\sup_{x\in 2X_{j}}\left(  1+L\left\vert
x-y\right\vert \right)  ^{-d-1}\\
&  =\sum_{j\in J}\frac{L^{d+1}}{N}\left(  1+L\left\vert q_{j}-y\right\vert
\right)  ^{-d-1}\leq\sum_{j\in J}\frac{L^{d+1}}{N}\left(  1+\frac{L}%
{2}\left\vert p_{j}-y\right\vert \right)  ^{-d-1}\\
&  =\sum_{j\in J}\int_{R_{j}}L^{d+1}\left(  1+\frac{L}{2}\left\vert
p_{j}-y\right\vert \right)  ^{-d-1}d\mu\left(  x\right)  \\
&  \leq\sum_{j\in J}\int_{R_{j}}L^{d+1}\left(  1+\frac{L}{2}\left\vert
x-y\right\vert \right)  ^{-d-1}d\mu\left(  x\right)  \\
&  \leq\int_{\mathcal{M}  }L^{d+1}\left(
1+\frac{L}{2}\left\vert x-y\right\vert \right)  ^{-d-1}d\mu\left(  x\right)
\\
&  \leq c_6L^{d+1}\int_{0}^{+\infty}\left(  1+\frac{L}{2}s\right)
^{-d-1}s^{d-1}ds\\
&  \leq c_6L^{d+1}\left(  \int%
_{0}^{1/L}s^{d-1}ds+\left(  \frac{2}{L}\right)  ^{d+1}\int%
_{1/L}^{+\infty}s^{-2}ds\right)  \\
&  \leq (d^{-1}+2^{d+1}) c_{6}L,
\end{align*}
Observe that the cardinality of $J^{\prime}$ is uniformly bounded with respect to
$y$ and $N$. Indeed, the cardinality of $J'$ is bounded above by the number of
inner balls $Y_j$ that are contained in the ball $B(y, 4\delta)$, and this number is 
bounded above by the ratio
\[
\frac{\mu(B(y,4\delta))}{\min_{j=1,\ldots, N}\mu(Y_j)}\le
\frac{c_5(8c_2N^{-1/d})^d}{c_4(c_1N^{-1/d})^d}=\frac{c_5(8c_2)^d}{c_4c_1^d}.
\]
Therefore, since $L\le N^{1/d}$,
\begin{align*}
\sum_{j\in J^{\prime}}\frac{L^{d+1}}{N}\sup_{x\in 2X_{j}}\left(  1+L\left\vert
x-y\right\vert \right)  ^{-d-1}&\leq\sum_{j\in J^{\prime}}\frac{L^{d+1}}{N}
\leq
\frac{c_5(8c_2)^d}{c_4c_1^d}\frac{L^{d+1}}{N}
\leq
\frac{8^dc_5c_2^d}{c_4c_1^d}L.
\end{align*}
Overall%
\begin{align*}
&\left\vert \int_{\mathcal{M}}\left\Vert \nabla P\left(  x\right)  \right\Vert
d\mu\left(  x\right)  -\sum_{j=1}^{N}\frac{1}{N}\left\Vert \nabla P\left(
x_{j}\right)  \right\Vert \right\vert \\
\leq&\left(  2+\delta\right)
\varepsilon+\kappa\left((d^{-1}+2^{d+1})c_6+\frac{8^dc_5c_2^d}{c_4c_1^d}\right)\delta L\int_{\mathcal{M}}\left\Vert \nabla P\left(  y\right)
\right\Vert d\mu\left(  y\right)  .
\end{align*}
Taking
\[
\varepsilon=\frac{\kappa\left((d^{-1}+2^{d+1})c_6+\dfrac{8^dc_5c_2^d}{c_4c_1^d}\right)\delta L}{2+\delta}\int_{\mathcal{M}}\left\Vert \nabla P\left(
y\right)  \right\Vert d\mu\left(  y\right)
\]
we obtain%
\[
\left\vert \int_{\mathcal{M}}\left\Vert \nabla P\left(  x\right)  \right\Vert
d\mu\left(  x\right)  -\sum_{j=1}^{N}\frac{1}{N}\left\Vert \nabla P\left(
x_{j}\right)  \right\Vert \right\vert \leq C_3LN^{-1/d}\int_{\mathcal{M}%
}\left\Vert \nabla P\left(  y\right)  \right\Vert d\mu\left(  y\right),
\]
where
\[
C_3=4c_2\kappa\left((d^{-1}+2^{d+1})c_6+\frac{8^dc_5c_2^d}{c_4c_1^d}\right).
\]
It remains to show (\ref{kernel estimate}). Since%
\[
\nabla_{y}S_{x}T_{x}W_{L}\left(  x,y\right)  =\sum_{0<\lambda_{k}\leq2L}%
\frac{1}{\lambda_{k}^{2}}h\left(  \frac{\lambda_{k}}{L}\right)  S%
T\varphi_{k}\left(  x\right)  \nabla\varphi_{k}\left(  y\right)  ,
\]
it is enough to estimate%
\[
\sum_{0<\lambda_{k}\leq2L}\frac{1}{\lambda_{k}^{2}}h\left(  \frac{\lambda_{k}%
}{L}\right)  ST\varphi_{k}\left(  x\right)  U\varphi_{k}\left(
y\right)
\]
for a generic vector field $U$ with $\left\Vert U\left(  x\right)  \right\Vert
=1$ for all $x\in\mathcal{M}$. Since for all $u>0$ $h(u)=\sum_{j=0}^{+\infty}h(2^ju)-h(2^{j+1}u)$, we have%
\begin{align*}
&  \sum_{0<\lambda_{k}}\frac{1}{\lambda_{k}^{2}}h\left(  \frac
{\lambda_{k}}{L}\right)  ST\varphi_{k}\left(  x\right)  U%
\varphi_{k}\left(  y\right)  \\
&  =\sum_{j=0}^{+\infty}\left(  \sum_{0<\lambda_{k}}\frac
{1}{\lambda_{k}^{2}}\left(  h\left(  \frac{2^{j}\lambda_{k}}{L}\right)
-h\left(  \frac{2^{j+1}\lambda_{k}}{L}\right)  \right)  ST\varphi
_{k}\left(  x\right)  U\varphi_{k}\left(  y\right)  \right)  .
\end{align*}
Now apply Theorem \ref{kernel} with $\phi_{k}=ST\varphi_{k}$,
$\psi_{k}=U\varphi_{k}$, $\kappa_1=c_7(2,2)$, $\kappa_2=c_7(1,1)$, $\kappa_3=c_7(2,1)$, 
$\kappa_4=c_8(2,1)$, $A_{1}=d+4$,
$A_{2}=d+2,$ $A_3=(d+3)/2$, $S>d+3$, $H\left(  u\right)  =u^{-2}\left(  h\left(  2u\right)
-h\left(  4u\right)  \right)  \in\mathcal{C}^{S}$ and is supported in $\{1/4\le|u|\le1\}$. Finally replace $L$ in
Theorem \ref{kernel} with $2L/2^{j}.$ Thus for $x,y\in\mathcal{M}$,%
\begin{align*}
&  \left\vert \sum_{0<\lambda_{k}}\frac{1}{\lambda_{k}^{2}}\left(  h\left(
\frac{2^{j}\lambda_{k}}{L}\right)  -h\left(  \frac{2^{j+1}\lambda_{k}}%
{L}\right)  \right)  ST\varphi_{k}\left(  x\right)  U\varphi
_{k}\left(  y\right)  \right\vert \\
&  =\left(  \frac{2L}{2^{j}}\right)  ^{-2}\left\vert \sum_{0<\lambda_{k}}\left(
\frac{2L}{2^{j}\lambda_{k}}\right)  ^{2}\left(  h\left(  \frac{2^{j}\lambda
_{k}}{L}\right)  -h\left(  \frac{2^{j+1}\lambda_{k}}{L}\right)  \right)
ST\varphi_{k}\left(  x\right)  U\varphi_{k}\left(  y\right)
\right\vert \\
&  =\left(  \frac{2L}{2^{j}}\right)  ^{-2}\left\vert \sum_{0<\lambda_{k}}H\left(
\frac{2^{j}\lambda_{k}}{2L}\right)  ST\varphi_{k}\left(  x\right)
U\varphi_{k}\left(  y\right)  \right\vert \\
&  \leq C_9\left(  \frac{2L}{2^{j}}\right)  ^{-2}\frac{\left( 2 L/2^{j}\right)
^{d+3}}{\left(  1+2L{2^{-j}}\left\vert x-y\right\vert \right)  ^{S}}=C_9\frac{\left(
2L2^{-j}\right)  ^{d+1}}{\left(  1+2L2^{-j}\left\vert x-y\right\vert \right)
^{S}}.
\end{align*}
Adding up in $j,$%
\begin{align*}
&  \left\vert \sum_{j=0}^{+\infty}\left(  \sum_{0<\lambda_{k}%
}\frac{1}{\lambda_{k}^{2}}\left(  h\left(  \frac{2^{j}\lambda_{k}}{L}\right)
-h\left(  \frac{2^{j+1}\lambda_{k}}{L}\right)  \right)  ST\varphi
_{k}\left(  x\right)  U\varphi_{k}\left(  y\right)  \right)  \right\vert
\\
&  \leq C_9\sum_{j=0}^{+\infty}\frac{\left(  2L2^{-j}\right)
^{d+1}}{\left(  1+2L2^{-j}\left\vert x-y\right\vert \right)  ^{S}}.
\end{align*}
Now, if $L\left\vert x-y\right\vert \leq1$ then%
\[
\sum_{j=0}^{+\infty}\frac{\left( 2 L2^{-j}\right)  ^{d+1}%
}{\left(  1+2L2^{-j}\left\vert x-y\right\vert \right)  ^{S}}\leq \sum
_{j=0}^{+\infty}\left(  2L2^{-j}\right)  ^{d+1}\leq 2^{d+2}L^{d+1},
\]
while if $L\left\vert x-y\right\vert \geq1$ then%
\begin{align*}
&  \sum_{j=0}^{+\infty}\frac{\left(  2L2^{-j}\right)  ^{d+1}%
}{\left(  1+2L2^{-j}\left\vert x-y\right\vert \right)  ^{S}}\\
&  \leq\sum_{1\leq2^{j}\leq L\left\vert x-y\right\vert }\frac{\left(
2L2^{-j}\right)  ^{d+1}}{\left(  1+2L2^{-j}\left\vert x-y\right\vert \right)
^{S}}+\sum_{L\left\vert x-y\right\vert \leq2^{j}}%
\frac{\left(  2L2^{-j}\right)  ^{d+1}}{\left(  1+2L2^{-j}\left\vert
x-y\right\vert \right)  ^{S}}\\
&  \leq \sum_{1\leq2^{j}\leq L\left\vert x-y\right\vert }\frac{\left(
2L2^{-j}\right)  ^{d+1}}{\left(  2L2^{-j}\left\vert x-y\right\vert \right)
^{S}}+\sum_{L\left\vert x-y\right\vert \leq2^{j}}\left(
2L2^{-j}\right)  ^{d+1}\\
&  \leq 2^{d+2-S}\frac{L^{d+1-S}}{\left\vert x-y\right\vert ^{S}}L^{S-d-1}\left\vert
x-y\right\vert ^{S-d-1}+2^{d+2}L^{d+1}L^{-d-1}\left\vert x-y\right\vert ^{-d-1}\\
&  \leq 2^{d+3}\left\vert x-y\right\vert ^{-d-1}.
\end{align*}
Overall%
\[
\sum_{j=0}^{+\infty}\frac{\left(  L2^{-j}\right)  ^{d+1}%
}{\left(  1+L2^{-j}\left\vert x-y\right\vert \right)  ^{S}}\leq 2^{d+3}L^{d+1}%
\left(  1+L\left\vert x-y\right\vert \right)  ^{-d-1}.
\]

\section{Proof of Theorem \ref{final}}\label{proof}
Let $\Omega$ be the open subset of the vector space $\Pi_L^0\subset L^2(\mathcal M,\,d\mu)$
\[
\Omega=\left\{P\in\Pi_L^0:\int_{\mathcal M}\|\nabla P(x)\|d\mu(x)<1\right\}.
\]
Since $\int_{\mathcal M}\|\nabla P(x)\|d\mu(x)$ is a norm in the finite dimensional space $\Pi_L^0$, it is equivalent to the $L^2$ norm 
in $\Pi_L^0$, and it follows that $\Omega$ is bounded in $\Pi_L^0\subset L^2(\mathcal M)$, and the map from $\Pi_L^0\subset L^2(\mathcal M)$ to $\mathbb R$ given by 
\[
P\mapsto \int_{\mathcal M}\|\nabla P(x)\|d\mu(x),
\]
is continuous, so that $\Omega$ is open.
\begin{lemma}\label{Fcontinua}
There exists a continuous map $F:\Pi_L^0\to\mathcal M^N$ such that for every $P\in\partial\Omega$
\[
\sum_{j=1}^NP(x_j(P))>0,
\]
 where  $F(P)=(x_1(P),\ldots,x_N(P))$.
\end{lemma}

Let us first show that this lemma readily implies Theorem \ref{final}. 
By the Riesz representation theorem, for each point $x\in\mathcal M$
there exists a unique polynomial $G_x\in\Pi_L^0$ such that 
\[
\langle G_x, P\rangle = P(x) \text{ for all } P\in \Pi_L^0.
\]
Then a set of points $x_1,\ldots,x_N\in\mathcal M$ forms an $L$-design  if and only if
\[
G_{x_1}+\cdots+G_{x_N}=0.
\]
Now let $Z:\mathcal M^N\to\Pi_L^0$ be the continuous map defined by
\[
Z(x_1,\ldots,x_N)=G_{x_1}+\cdots+G_{x_N},
\]
and call $f=Z\circ F:\Pi_L^0\to\Pi_L^0$. Clearly for every $P\in\partial \Omega$ we have
\[
\langle P,f(P)\rangle=\sum_{j=1}^NP(x_j(P))>0
\]
by the lemma, and by Theorem \ref{brouwer} it follows that there exists 
$Q\in\Omega$ such that $Z(F(Q))=0$, that is such that $G_{x_1(Q)}+\cdots +G_{x_N(Q)}=0$ which implies that $\{x_1(Q),\ldots, x_N(Q)\}$
is an $L$-design.

\begin{proof}[Proof of Lemma \ref{Fcontinua}]
Take a partition of $\mathcal M$ as in Theorem \ref{partition} with
constants $c_1$ and $c_2$ and let $C_{\mathcal M}\ge\max\{1, 2^d C_3(c_1,2c_2)^d, 2^d C_3(c_1,13c_2)^d\}$,
where the constants $C_3(\cdot,\cdot)$ are as in Theorem \ref{MZ gradients}.
For each $j=1,\ldots,N$ choose an arbitrary $x_j\in R_j$.

Now fix $\varepsilon<1/4$ and let as before $v_{\varepsilon}:\mathbb{[}0,\mathbb{+\infty
)}\rightarrow\mathbb{R}$ be a $\mathcal C^\infty$ function such that $v_{\varepsilon
}\left(  u\right)  =u$ for $u\geq\varepsilon$, $v_{\varepsilon}\left(
u\right)  =\varepsilon/2$ for $u<\varepsilon/4$ and $v_{\varepsilon}\left(
u\right)  \geq u$ for all $u\geq0$. 

Take the mapping $U:\Pi_L^0\to \mathcal X(\mathcal M)$
defined by
\[
U(P)(y)=\frac{\nabla P(y)}{v_{\varepsilon}(\|\nabla P(y)\|)}, \quad y\in\mathcal M.
\]
For each $j=1,\ldots,N$ let $y_j:\Pi_L^0\times [0,+\infty)\to\mathcal M$
be the map satisfying the differential equation
\[
\frac {d}{dt}y_j(P,t)=U(P)(y_j(P,t))
\]
with the initial condition 
\[
y_j(P,0)=x_j
\]
for each $P\in\Pi_L^0$.
Since the mapping $U(p,y)$ is Lipschitz continuous in both $P$ and $y$, each $y_j$
is well defined and continuous in $P$ and $t$. 
Now set
\[
F(P)=(x_1(P),\ldots,\,x_N(P)):=\left(y_1\left(P,12c_2N^{-1/d}\right),\ldots,\,y_N\left(P,12c_2N^{-1/d}\right)\right),
\]
which is continuous on $\Pi_L^0$ by definition. Let now $P\in\partial\Omega$, that is 
\[
\int_{\mathcal M}\|\nabla P(x)\|d\mu(x)=1.
\]
We have
\begin{align*}
\frac 1N\sum_{j=1}^N P(x_j(P))&=\frac 1N\sum_{j=1}^N P(y_j(P,12c_2N^{-1/d}))\\
&=\frac 1N\sum_{j=1}^N P(x_j)+\int_{0}^{12c_2N^{-1/d}}\frac{d}{dt}\left(\frac 1N\sum_{j=1}^N P(y_j(P,t))\right)dt.
\end{align*}
Now,
\begin{align*}
\left|\frac 1N\sum_{j=1}^N P(x_j)\right|&=\left|\sum_{j=1}^N\int_{R_j}(P(x_j)-P(x))d\mu(x)\right|
\le
\sum_{j=1}^N\int_{R_j}|P(x_j)-P(x)|d\mu(x)\\
&\le\frac1{N}\sum_{j=1}^N\mathrm{diam}(R_j)\max_{z\in 2X_j}
\|\nabla P(z)\|
\le\frac{2c_2}{N^{1+1/d}}\sum_{j=1}^N
\|\nabla P(z_j)\|,
\end{align*}
where $z_j$ is the point that realizes the maximum.
Observe that the partition $\mathcal R'=\{R_1',\ldots,\,R_N'\}$ defined by $R_j'=R_j\cup\{z_j\}$ 
satisfies Theorem \ref{partition} with constants $c_1$ and $2c_2$. Therefore, by Theorem \ref{MZ gradients} applied to $P$ and the partition $\mathcal R'$, since
$C_3(c_1,2c_2)LN^{-1/d}\le(C_{\mathcal M}L^dN^{-1})^{1/d}/2\le1/2$, we have 
\begin{align*}
&\left|\frac 1N\sum_{j=1}^N P(x_j)\right|\le 
\frac{2c_2}{N^{1+1/d}}\sum_{j=1}^N\|\nabla P(z_j)\|\\
\le & \frac{2c_2}{N^{1/d}}\left|\sum_{j=1}^N\frac 1N\|\nabla P(z_j)\|-\int_{\mathcal M}\|\nabla P(z)\|d\mu(z)\right|+
\frac{2c_2}{N^{1/d}}\int_{\mathcal M}\|\nabla P(z)\|d\mu(z)\\
&\le \frac{3c_2}{N^{1/d}}\int_{\mathcal M}\|\nabla P(z)\|d\mu(z)=\frac{3c_2}{N^{1/d}}
\end{align*}
for any $P\in\partial\Omega$. On the other hand, for $t\in[0,12c_2N^{-1/d}]$
\begin{align*}
\frac{d}{dt}\left(\frac 1N\sum_{j=1}^N P(y_j(P,t))\right)=&\frac 1N\sum_{j=1}^N\frac{\|\nabla P(y_j(P,t))\|^2}{v_{\varepsilon}(\|\nabla P(y_j(P,t))\|)}\\
\ge &\frac 1N\sum_{j:\|\nabla P(y_j(P,t))\|\ge\varepsilon}\|\nabla P(y_j(P,t))\|\\
\ge&\frac 1N\sum_{j=1}^N\|\nabla P(y_j(P,t))\|-\varepsilon.
\end{align*}
Since clearly $|y_j(P,t)-x_j|\le t$, the partition $\mathcal R''=\{R_1'',\ldots,\,R_N''\}$ defined by $R_j''=R_j\cup\{y_j(P,t)\}$ 
satisfies Theorem \ref{partition} with the constants $c_1$ and $13c_2$. Therefore by Theorem \ref{MZ gradients} applied to $P$ and the partition $\mathcal R''$,
since $C_3(c_1,13c_2)LN^{-1/d}\le(C_{\mathcal M}L^dN^{-1})^{1/d}/2\le1/2$, we have
\begin{align*}
&\frac{d}{dt}\left(\frac 1N\sum_{j=1}^N P(y_j(P,t))\right)\\
\ge&\frac 1N\sum_{j=1}^N\|\nabla P(y_j(P,t))\|-\varepsilon\\
\ge&\int_{\mathcal M}\|\nabla P(y)\|d\mu(y)-\left|\int_{\mathcal M}\|\nabla P(y)\|d\mu(y)-\frac 1N\sum_{j=1}^N\|\nabla P(y_j(P(t)))\|\right|-\varepsilon\\
\ge&\frac 12\int_{\mathcal M}\|\nabla P(y)\|d\mu(y)-\varepsilon=\frac 12-\varepsilon
\end{align*}
for each $P\in\partial\Omega$ and $t\in[0,12c_2N^{-1/d}]$. Finally,
\begin{align*}
\frac 1N\sum_{j=1}^N P(x_j(P))&=\frac 1N\sum_{j=1}^N P(x_j)+\int_{0}^{12c_2N^{-1/d}}\frac{d}{dt}\left(\frac 1N\sum_{j=1}^N P(y_j(P,t))\right)dt\\
&\ge \frac{12c_2}{N^{1/d}}\left(\frac 12-\varepsilon\right)-\frac{3c_2}{N^{1/d}}=(3-12\varepsilon)\frac{c_2}{N^{1/d}}>0.
\end{align*}

\end{proof}


\begin{thebibliography}{999999}     

\bibitem{BGM} M. Berger, P. Gauduchon, E. Mazet, Le spectre d'une vari\'et\'e riemannienne, \textit{Lecture Notes in Mathematics,} \textbf{194}, Springer-Verlag, Berlin-New York, 1971.
                                                                                      %
\bibitem{BRV} A. Bondarenko, D. Radchenko, M. Viazovska, \textit{Optimal asymptotic bounds for 
spherical designs}, Ann. Math. \textbf{178}, 443--452 (2013).

\bibitem{quadrature} L. Brandolini, C. Choirat, L. Colzani, G. Gigante, R. Seri, G. Travaglini,
\textit{Quadrature rules and distribution of points on manifolds}, Ann. Sc. Norm. Sup. Pisa Cl. Sci. (5)
\textbf{XIII}, 889--923 (2014).

\bibitem{Chen} L. Brandolini, W. W. L Chen, L. Colzani, G. Gigante, G. Travaglini, \textit{Discrepancy and Numerical Integration on Metric Measure Spaces}, J. Geom. Anal. (2018), https://doi.org/10.1007/s12220-018-9993-6.

\bibitem{Brauchart1} J. S. Brauchart, J. Dick, E. B. Saff, I. H. Sloan, Y. G. Wang, R. S. Womersley, \textit{Covering of spheres by spherical caps and worst-case error for equal weight cubature in Sobolev spaces}, J. Math. Anal. Appl. \textbf{431}, 782--811 (2015).

\bibitem{Brauchart2} J. S. Brauchart, E. B. Saff, I. H. Sloan, R. S. Womersley, \textit{QMC designs: optimal order quasi Monte Carlo integration schemes on the sphere}, Math. Comput. \textbf{83}, 2821--2851 (2014).

\bibitem{DC} M. P. do Carmo, Riemannian Geometry, Translated from the second Portuguese edition by Francis Flaherty. \textit{Mathematics: Theory \& Applications,} Birkh\"auser Boston, Inc., Boston, MA, 1992.

\bibitem{EMOC} U. Etayo, J. Marzo, J. Ortega-Cerd\`a, \textit{Asymptotically optimal designs on 
compact algebraic manifolds}, J. Monatsh. Math. \textbf{186}, 235--248 (2018).

\bibitem{FM2010}F. Filbir, H. N. Mhaskar, \textit{A Quadrature Formula
for Diffusion Polynomials Corresponding to a Generalized Heat Kernel,} J.
Fourier Anal. Appl. \textbf{16}, 629--657 (2010).

\bibitem{FM2011}F. Filbir, H. N. Mhaskar, \textit{Marcinkiewicz-Zygmund measures on manifolds,}
J. Complexity \textbf{27}, 568--596 (2011).

\bibitem{GL}G. Gigante, P. Leopardi, \textit{Diameter bounded equal
measure partitions of Ahlfors regular metric measure spaces,} Discret. Comput.
Geom. \textbf{57}, 419--430 (2017).

\bibitem{Gr} P. Greiner, \textit{An asymptotic expansion for the heat equation,} Arch. Rational Mech. Anal. \textbf{41}, 163--218 (1971).


\bibitem{HOR}L. Hormander, The analysis of linear partial differential
operators, I II III IV, Springer Verlag, 1983-1985.

\bibitem{Ko} Yu. A. Kordyukov, \textit{$L^p$ theory of elliptic differential operators on manifolds of bounded geometry,}  Acta Appl. Math. \textbf{23}, 223--260 (1991).

\bibitem{KM}J. Korevaar, J. L. H. Meyers, \textit{Spherical Faraday cage for the case of equal
point charges and Chebyshev-type quadrature on the sphere,} Integral Transform. Spec. Funct. \textbf{1},
105--117 (1993).

\bibitem{L} J. M. Lee, Introduction to smooth manifolds,
Second edition, \textit{Graduate Texts in Mathematics,} \textbf{218}, Springer, New York, 2013.

\bibitem{MM}M. Maggioni, H. N. Mhaskar, \textit{Diffusion polynomial frames on metric measure spaces,} Appl. Comput. Harmon. Anal. \textbf{24}, 329--353 (2008).

\bibitem{brouwer} D. O'Regan, Y. J. Cho, Y.-Q. Chen, Topological Degree Theory and Applications,
\textit{Ser. Math. Anal. Appl.}, \textbf{10}, Chapman \& Hall / CRC, Boca Raton, FL, 2006.

\bibitem{Sikora}A. Sikora, Riesz transform, \textit{Gaussian bounds
and the method of wave equation}. Math. Z. \textbf{247}, 643--662 (2004).

\bibitem{Stein}E. M. Stein, R. Shakarchi, Complex analysis. \textit{Princeton Lectures in Analysis, 2}. Princeton University Press, Princeton, NJ, 2003.                                                                       


\end{thebibliography}
\end{document}